\documentclass[]{amsart}
\usepackage[utf8]{inputenc}
\usepackage{amsmath}
\usepackage{amsfonts}
\usepackage{amsthm}
\usepackage{amssymb}
\usepackage{enumerate}

\newcommand{\ts}{\textsuperscript}

\newcommand{\R}{\mathbb{R}}

\newcommand{\Z}{\mathbb{Z}}

\newcommand{\N}{\mathbb{N}}
\renewcommand{\P}{\mathbb{P}}

\renewcommand{\le}{\leqslant}
\renewcommand{\ge}{\geqslant}
\newcommand{\quotient}[2]{{\raisebox{.2em}{$#1$}\left/\raisebox{-.2em}{$#2$}\right.}}

\theoremstyle{plain}% default
\newtheorem{thm}{Theorem}[section]
\newtheorem{lem}[thm]{Lemma}
\newtheorem{prop}[thm]{Proposition}
\newtheorem{cor}[thm]{Corollary}

\newtheorem{thmR}{Theorem}
\newtheorem{propR}[thmR]{Proposition}

\makeatletter
\newtheorem*{rep@theorem}{\rep@title}
\newcommand{\newreptheorem}[2]{%
\newenvironment{rep#1}[1]{%
 \def\rep@title{#2 \ref{##1}}%
 \begin{rep@theorem}}%
 {\end{rep@theorem}}}
\makeatother
\newreptheorem{theorem}{Theorem}

\makeatletter
\newtheorem*{rep@proposition}{\rep@title}
\newcommand{\newrepproposition}[2]{%
\newenvironment{rep#1}[1]{%
 \def\rep@title{#2 \ref{##1}}%
 \begin{rep@proposition}}%
 {\end{rep@proposition}}}
\makeatother
\newreptheorem{proposition}{Proposition}

\makeatletter
\newtheorem*{rep@corollary}{\rep@title}
\newcommand{\newrepcorollary}[2]{%
\newenvironment{rep#1}[1]{%
 \def\rep@title{#2 \ref{##1}}%
 \begin{rep@corollary}}%
 {\end{rep@corollary}}}
\makeatother
\newreptheorem{corollary}{Corollary}

\newtheorem*{thm*}{Theorem}
\newtheorem*{not*}{Notation}

\theoremstyle{definition}
\newtheorem*{defn*}{Definition}
\newtheorem{defn}[thm]{Definition}
\newtheorem{conj}[thm]{Conjecture}

\theoremstyle{plain}
\newtheorem{rem}[thm]{Remark}
\newtheorem*{remR}{Remark}

\DeclareMathOperator\Supp{supp}
\DeclareMathOperator\FSym{FSym}
\DeclareMathOperator\FAlt{FAlt}
\DeclareMathOperator\Sym{Sym}
\DeclareMathOperator\Aut{Aut}
\DeclareMathOperator\Inn{Inn}
\DeclareMathOperator\Stab{Stab}

\newcommand{\id}{\mathrm{id}}

\begin{document}
\title{A note on the $R_\infty$ property for groups $\FAlt(X)\leqslant G\leqslant\Sym(X)$}
%\title{The $R_\infty$ property and groups $\FAlt(X)\leqslant G\leqslant\Sym(X)$}
\author{Charles Garnet Cox}
\address{Mathematical Sciences, University of Southampton, SO17 1BJ, UK}
\email{cpgcox@gmail.com}
%\urladdr{http://www.southampton.ac.uk/~cc14v07/}
\thanks{}



\date{June, 2018}
\begin{abstract}
Given a set $X$, the group $\Sym(X)$ consists of all bijective maps from $X$ to $X$, and $\FSym(X)$ is the subgroup of maps with finite support i.e.\ those that move only finitely many points in $X$. We describe the automorphism structure of groups $\FSym(X)\le G\le \Sym(X)$ and use this to state some conditions on $G$ for it to have the $R_\infty$ property. Our main results are that if $G$ is infinite, torsion, and $\FSym(X)\le G\le \Sym(X)$, then it has the $R_\infty$ property. Also, if $G$ is infinite and residually finite, then there is a set $X$ such that $G$ acts faithfully on $X$ and, using this action, $\langle G, \FSym(X)\rangle$ has the $R_\infty$ property. Finally we have a result for the Houghton groups, which are a family of groups we denote $H_n$, where $n \in \N$. We show that, given any $n\in \N$, any group commensurable to $H_n$ has the $R_\infty$ property.
\end{abstract}
\maketitle
\section{Introduction}
The notion of twisted conjugacy and its relationship to fixed point theory has attracted significant attention. For any group $G$ and any $\phi \in \Aut(G)$, we say that two elements $a, b \in G$ are $\phi$-twisted conjugate (denoted $a\sim_\phi b$) if there exists an $x \in G$ such that
\begin{align}\label{eqntwistedconj}
 (x^{-1})\phi ax = b.
\end{align}
Notice that when $\phi=\id_G$ this becomes the equation for conjugacy. Now, given any $\phi \in \Aut(G)$, define  the Reidemeister number of $\phi$, denoted $R(\phi)$, to be the number of $\phi$-twisted conjugacy classes in $G$. Thus $R(\id_G)$ records the number of conjugacy classes of $G$ and deciding whether this is infinite has been studied for some time (e.g.\ \cite{HNNext} where an infinite group with $R(\id_G)$ finite was constructed). We say that $G$ has the $R_{\infty}$ property if $R(\phi)=\infty$ for every $\phi \in \Aut(G)$.

\begin{not*} For a non-empty set $X$, let $\Sym(X)$ denote the group of all permutations of $X$. Furthermore, let $\FSym(X)$ denote the group of all permutations of $X$ with finite support, and let $\FAlt(X)$ denote the group of all even permutations of $X$ with finite support.
\end{not*}

A first example one may consider for the $R_\infty$ property is $\Z$. Although this has infinitely many conjugacy classes, the only non-trivial automorphism has Reidemeister number 2. Similarly, for any $m \in \N:=\{1, 2, \ldots\}$, the automorphism $\psi$ of $\Z^m$ which sends $a$ to $a^{-1}$ for all $a \in \Z^m$ has Reidemeister number $2^m$. In \cite{rinfinityHou} and \cite{rinfinityHou2} however, the family of Houghton groups, which (for any $n \in \N$) are denoted $H_n$, act on $\{1,\ldots,n\}\times \N=:X_n$, and which lie in the short exact sequence
 $$
1\longrightarrow \FSym(X_n)\stackrel{}{\longrightarrow} H_n\stackrel{}{\longrightarrow} \Z^{n-1}\longrightarrow 1
 $$
were shown to have the $R_\infty$ property. In this note we start with a simpler, more general proof of their theorem, and then develop this in various directions.
\begin{defn} A group $G$ \emph{fully contains} $\FAlt(X)$ if $\FAlt(X)\leqslant G\leqslant \Sym(X)$. Since we only wish to investigate infinite groups, we will always consider $X$ to be infinite. We do not, however, place any other cardinality assumptions on $X$.
\end{defn}
Note that any Houghton group $H_n$ fully contains $\FAlt(X_n)$, but let us justify that this is a large class of groups, using a construction from \cite{osintransitive}. For any infinite group $G$, we have that $G\le\Sym(X)$ for some $X$ (with the possibility that $X=G$ since $G$ can always be embedded into $\Sym(G)$ using a regular representation of $G$). Then $\langle G, \FAlt(X)\rangle$ fully contains $\FAlt(X)$.
\begin{conj} Let $G$ be an infinite group that acts faithfully on a set $X$. Then $\langle G, \FAlt(X)\rangle$ and $\langle G, \FSym(X)\rangle$ both have the $R_\infty$ property.
\end{conj}

We make some progress with this conjecture. We first confirm it for the case where $G$ is torsion i.e.\ we show that, for any infinite set $X$ and any embedding of $\Psi: G\hookrightarrow\Sym(X)$, $\langle (G)\Psi, \FAlt(X)\rangle$ has the $R_\infty$ property for any torsion group $G$. We then use this work to show that if $G$ is an infinite residually finite group, then there is a set $X$ on which $G$ acts faithfully and, using this action, $\langle G, \FAlt(X)\rangle$ has the $R_\infty$ property. We end by showing that, if $G$ is commensurable to a Houghton group $H_n$ (where $n\in\N$) then $G$ has the $R_\infty$ property.

Let us now describe these results more precisely, and better indicate the path that the paper takes. We start by describing the automorphism group for groups fully containing $\FAlt(X)$, so to approach twisted conjugacy.

\begin{defn*}  A group $G$ is monolithic if it has a non-trivial normal subgroup that is contained in every non-trivial normal subgroup of $G$ i.e.\ if it has a minimal non-trivial normal subgroup.
\end{defn*}
Let $N_{\Sym(X)}(G):=\{\rho \in \Sym(X)\mid\rho^{-1}g\rho\in G\;\text{for every}\;g\in G\}$, the normaliser of $G$ over $\Sym(X)$.
\begin{propR} (Lem. \ref{lem-autresult}, Prop. \ref{prop-faltchar}, Rem. \ref{FSymchar}). Let $G$ fully contain $\FAlt(X)$, where $X$ is infinite. Then $\FAlt(X)$ is characteristic in $G$, $\Aut(G)\cong N_{\Sym(X)}(G)$, and $G$ is monolithic. Moreover, since elements of $\Aut(G)$ preserve the cycle type of elements of $G$, if $\FSym(X)\le G$, then $\FSym(X)$ is characteristic in $G$.
\end{propR}
The following well known lemma implies that if $G$ is a group with the $R_\infty$ property and $G$ acts faithfully on a set $X$, then $\langle G, \FAlt(X)\rangle$ and $\langle G, \FSym(X)\rangle$ have the $R_\infty$ property (since Proposition 1 states that $\FAlt(X)$ and $\FSym(X)$ are characteristic in $\langle G, \FAlt(X)\rangle$).

\begin{lem}\cite[Lem 2.1]{rinfinityextensions}\label{lem-finiteextensions} For any short exact sequence of groups
 $$
1\longrightarrow D \longrightarrow  E \longrightarrow F \longrightarrow 1
 $$
if $D$ is characteristic in $E$ and $F$ has the $R_{\infty}$ property, then $E$ has the $R_{\infty}$ property.
\end{lem}

We then work with arguments using cycle type (using that the conjugacy classes of $\Sym(X)$ are well know: each consists of all elements of the same cycle type).
\begin{defn*} Let $g \in \Sym(X)$. Then an orbit of $g$ is $\{xg^d\mid d \in \Z\}$ where $x \in X$. Also, $g$ has an infinite orbit if there is a $y \in X$ such that $\{yg^d\mid d \in \Z\}$ is infinite.
\end{defn*}
\begin{repproposition}{firstresult} Let $G$ fully contain $\FAlt(X)$, where $X$ is an infinite set. If for every $\rho \in N_{\Sym(X)}(G)$ there is an $s \in \N$ such that $\rho$ has finitely many orbits of size $s$, then $G$ has the $R_\infty$ property.
\end{repproposition}
From the structure of $\Aut(H_n)$, where $H_n$ denotes the $n$\ts{th} Houghton group, Proposition \ref{firstresult} immediately yields that, for any $n \ge2$, $H_n$ has the $R_\infty$ property.

\begin{repcorollary}{prop-infinitecycleneeded} Let $G$ fully contain $\FAlt(X)$, where $X$ is an infinite set. If for every $g \in G$, $g$ does not have an infinite orbit, then $G$ has the $R_\infty$ property.
\end{repcorollary}
Clearly torsion groups satisfy Corollary \ref{prop-infinitecycleneeded}.
\begin{repcorollary}{cor-torsion} Let $G$ be an infinite torsion group which fully contains $\FAlt(X)$. Then $G$ has the $R_\infty$ property. 
\end{repcorollary}
This means that any torsion group $T$ can be embedded into an infinite torsion group (of any cardinality greater than or equal to $|T|$) which has the $R_\infty$ property. It is in fact easy to construct an uncountable family of such groups.
\begin{repcorollary}{cor-countabletorsion} There exist uncountably many countable torsion groups which have the $R_\infty$ property.
\end{repcorollary}
This result can be strengthened by using an already known family of countable, finitely generated, torsion groups.
\begin{repcorollary}{cor-countabletorsionfg} There exist uncountably many finitely generated torsion groups which have the $R_\infty$ property.
\end{repcorollary}
Residually finite groups are exactly those that have a faithful action on their finite quotients. This action therefore only has finite orbits, meaning that Corollary \ref{prop-infinitecycleneeded} applies.
\begin{repcorollary}{cor-residuallyfinite} Let $G$ be an infinite residually finite group, and $X$ be the union of the finite quotients of $G$. Then, using this action, $\langle G, \FAlt(X)\rangle$ has the $R_\infty$ property.
\end{repcorollary}

A few conventions will be used throughout this note:
\begin{enumerate}[i)]
\item we shall always work with right actions;
\item unless specified, $X$ will refer to an infinite set;
\item we shall always consider elements from $\Sym(X)$ to be written in disjoint cycle notation;
\item for all of the results in this note, the same proofs can be used if $\FAlt$ is replaced with $\FSym$.
\end{enumerate}

\begin{remR} Let $g \in \Sym(X)$. We shall say `a cycle of $g$' to refer, for some $x \in X$, to an orbit $\{xg^d\mid d \in \Z\}$. If there is an $x \in X$ such that this set is infinite, then this is an infinite cycle of $g$ and $g$ contains an infinite cycle. If there is an $x \in X$ such that this set has cardinality $r$, then this is an $r$-cycle of $g$ and $g$ contains an $r$-cycle. If, for some $s \in \N$, there are only finitely many $x \in X$ such that $|\{xg^d\mid d \in \Z\}|=s$, then we shall say that $g$ has finitely many $s$-cycles. Similarly $g$ may have finitely many infinite cycles.
\end{remR}
\noindent\textbf{Acknowledgements.} I thank the authors of \cite{rinfinityHou} and \cite{rinfinityHou2}, whose papers drew my attention to the $R_\infty$ property. I thank my supervisor Armando Martino for his guidance and encouragement. I thank Hector Durham, of the University of Southampton, for the numerous interesting discussions, especially those regarding monolithic groups. I thank Motiejus Valiunas, also of the University of Southampton, for his suggestion to choose $X$ as the union of finite quotients for an action with only finite orbits for residually finite groups, the main ingredient of Corollary \ref{cor-residuallyfinite}. Finally I thank the referee for their useful comments.

\section{Preliminary observations}
The groups $\FAlt(X)$, $\FSym(X)$, and $\Sym(X)$ often arise when considering permutation groups (see, for example, \cite{cameron} and \cite{Permutationgroups}). Note that any countable group can be considered as a subgroup of $\Sym(X)$ where $X$ is countable (for example set $X:=G$ and use the regular representation of $G$). Let us start by recalling some elementary observations about $\FAlt$.

\begin{lem} If $G$ fully contains $\FAlt(X)$ where $X$ is infinite, then $G$ is centreless.
\end{lem}

\begin{lem}\label{faltgenset} For any infinite set $X$, $\FAlt(X)$ is generated by $S$, where $S$ is the set of all 3-cycles with support in $X$.
\end{lem}
\begin{proof} Using $S$ we can produce any element which is a product of two 2-cycles (for example choose $(a_1\;a_2)(a_2\;b_1)$ and $(a_2\;b_1)(b_1\;b_2)$ whose product is $(a_1\;a_2)(b_1\;b_2)$). Now, given an element $\sigma \in \FAlt(X)$, write $\sigma$ as a product of 2-cycles. By definition this product will consist of an even number  of 2-cycles. Now, each pair of 2-cycles will either be: trivial; a 3-cycle; or a product of two 2-cycles.
\end{proof}
For any infinite set $X$ we therefore have that $|X|=|S|=|\FAlt(X)|$. Moreover $\FAlt(X)$ is an index 2 subgroup of $\FSym(X)$, and so for any infinite set $X$ we also have that $|X|=|\FSym(X)|$.
\begin{lem} For any infinite set $X$, $\FAlt(X)$ is simple.
\end{lem}
\begin{proof}
Assume that $1\ne\sigma \in N$, a non-trivial normal subgroup of $\FAlt(X)$. Then $\sigma\in A_n$ where $n\ge 5$. But $A_n\cap N\unlhd A_n$, and (since $N$ is non-trivial and $A_n$ is simple for $n\ge 5$) we have that $N\cap A_n=A_n$. Thus $N$ contains a 3-cycle and so $N=\FAlt(X)$ by the previous lemma.
\end{proof}
Note that no infinite simple group can be residually finite, and so if $G$ fully contains $\FAlt(X)$, then $G$ is not residually finite. Also, given any infinite set $X$, any group $G$ fully containing $\FAlt(X)$ will have $\FAlt(X)$ as a normal subgroup. Thus, unless $G=\FAlt(X)$, $G$ will not be simple. 

\begin{not*}
Let $G\leqslant \Sym(X)$. For any given $\rho \in N_{\Sym(X)}(G)$, let $\phi_\rho$ denote the automorphism of $G$ induced by conjugation by $\rho$ i.e.\ $(g)\phi_\rho:=\rho^{-1}g\rho$ for all $g \in G$.
\end{not*}
The three groups $\FAlt(X), \FSym(X), \Sym(X)$ have the property that
\begin{align}\label{autnorm}
N_{\Sym(X)}(G)\rightarrow \Aut(G), \rho\mapsto \phi_\rho\;\textrm{is an isomorphism.}
\end{align}
This means that $\Aut(\FAlt(X))\cong N_{\Sym(X)}(\FAlt(X))=\Sym(X)\cong \Aut(\FSym(X))$ and that $\FAlt(X)$ is characteristic in $\FSym(X)$ which is characteristic in $\Sym(X)$. Our first aim is to show that any group $G$ fully containing $\FAlt(X)$ satisfies (\ref{autnorm}). We do this by showing that $\FAlt(X)$ is characteristic in such a $G$ and then apply the following lemma.
\begin{lem}\label{lem-autresult} Let $G\leqslant \Sym(X)$ and $\FAlt(X)$ be a characteristic subgroup of $G$. Then $N_{\Sym(X)}(G)\cong_\Psi \Aut(G)$ where $\Psi: \rho\mapsto \phi_\rho$.
\end{lem}
\begin{proof} Running the proof of \cite[Cor. 3.3]{rinfinityHou2} using 3-cycles rather than 2-cycles yields the result. 
\end{proof}
For any group $G$ satisfying (\ref{autnorm}), we may use the following reformulation of twisted conjugacy, which has been used extensively by many authors working with the $R_\infty$ property. Recall that $\phi_\rho$ denotes the automorphism induced by conjugation by $\rho \in \Sym(X)$. Thus,
\begin{align}\label{eqnmanipulatetwistedconj}
 (x^{-1})\phi_\rho ax = b\Rightarrow \rho^{-1}(x^{-1})\rho ax = b \Rightarrow x^{-1}\rho ax = \rho b.
\end{align}
We may then show that $R(\phi_\rho)=\infty$ by finding a set of elements $\{a_k \in G\mid k \in \N\}$ such that
\begin{align}\label{eqnequivalenttwistedconj}
 \rho a_i \sim_G \rho a_j \Leftrightarrow i = j.
\end{align}

This is because, if such a set of elements exist, then each $a_k$ lies in a distinct $\phi_\rho$-twisted conjugacy class, and so $R(\phi_\rho)=\infty$. Thus showing, for each $\rho \in N_{\Sym(X)}(G)$, that there exists a set of elements $\{a_k \in G\mid k \in \N\}$ where (\ref{eqnequivalenttwistedconj}) holds is sufficient to show that $G$ has the $R_\infty$ property.

\begin{prop}\label{prop-faltchar} If $G$ fully contains $\FAlt(X)$, then $\FAlt(X)$ is a unique minimal normal subgroup of $G$. Moreover $\FAlt(X)$ is a characteristic subgroup of $G$.
\end{prop}
\begin{proof} We first show that $\FAlt(X)$ is a unique minimal normal subgroup of $G$, known as the monolithic property. Clearly $\FAlt(X)$ is normal in $G$, since it is normal in $\Sym(X)$ (conjugation in $\Sym(X)$ preserves cycle type).

Consider $N\unlhd G$. We have $N \cap \FAlt(X) \unlhd \FAlt(X)$, and since $\FAlt(X)$ is simple, $N \cap \FAlt(X)$ must either be trivial or $\FAlt(X)$. Let $g \in N\setminus\{1\}$. This must either: be in $\FSym(X)$; contain infinitely many finite cycles; or contain an infinite cycle. We now show that there exists a $\sigma \in \FAlt(X)$ such that $\sigma^{-1}g\sigma g^{-1} \in \FAlt(X)\setminus\{1\}$. Since $N$ is normal in $G$, $g$ and $\sigma^{-1}g\sigma$ are in $N$ and so this will prove the claim. For the case where $g \in \FSym(X)$, choose $\sigma$ so that $\sigma^{-1}g\sigma$ and $g$ have disjoint supports. For the case where $g$ contains infinitely many finite cycles, pick 4 distinct cycles (each of length greater than 1) of $g$ and points $b_1, b_2, b_3, b_4$: one from each cycle. A suitable $\sigma$ is then $(b_1\;b_2)(b_3\;b_4)$. Finally, assume that $g$ contains an infinite cycle. Let $x_0 \in X$ lie in some infinite cycle of $g$, and for every $i \in \Z$ let $x_i:=x_0g^i$.  Let $a:=(\ldots x_{-3}\;x_{-2}\;x_{-1}\;x_0\;x_1\;x_2\;x_3 \ldots)$ and let $\mu:=(x_{-1}\;x_0\;x_1)$. Straightforward computation shows that $\mu^{-1}a\mu a^{-1}=(x_{-2}\;x_{-1}\;x_1)$. Moreover, since $\mu$ commutes with $ga^{-1}$, we have that $\mu^{-1}g\mu g^{-1}=(x_{-2}\;x_{-1}\;x_1)$. Thus $\mu$ is a suitable candidate for $\sigma$ in this case.

Now, let $\phi \in \Aut(G)$ and consider $\FAlt(X) \cap (\FAlt(X))\phi$. As above, this must be trivial or $\FAlt(X)$. If it were trivial, this would contradict the uniqueness of $\FAlt(X)$ as a minimal, non-trivial, normal subgroup in $G$, and hence $\FAlt(X)$ is characteristic in $G$.
\end{proof}
\begin{rem} \label{FSymchar} We may use Lemma \ref{lem-autresult} and Proposition \ref{prop-faltchar} to prove that all automorphisms of $\Sym(X)$ are inner. Also, consider if $\FSym(X)\leqslant G\leqslant\Sym(X)$. Then, for all $\rho \in N_{\Sym(X)}(G)$ and all $g \in \FSym(X)$, we have that $(g)\phi_\rho$ has the same cycle type as $g$. Thus $\FSym(X)$ is characteristic in $G$.
\end{rem}

We are now ready to produce conditions on the cycle type of elements in $G$ and in $N_{\Sym(X)}(G)$ for automorphisms to have infinite Reidemeister number. In order to do this we will use the condition equivalent to showing that $R(\phi_\rho)=\infty$ (labelled (\ref{eqnequivalenttwistedconj}) above) and well known facts about $\Sym(X)$ regarding cycle type.
\section{Results using facts about conjugacy in $\Sym$}
\begin{lem}\label{innerautsfine} Let $Y$ be an infinite set and $X$ be an infinite subset of $Y$. If $\FAlt(X)$ is a subgroup of $G\le \Sym(Y)$, then $R(\id_G)=\infty$.
\end{lem}
\begin{proof} We produce an infinite family of elements in $G$ which all lie in distinct conjugacy classes. We have the equation $g^{-1}ag = b$. Conjugation by elements of $G$ cannot change the cycle type of elements of $\Sym(X)\le \Sym(Y)$. Thus choosing $a_k$ to be a cycle of length $2k+1$ (or any infinite family of elements of $\FAlt(X)$ with distinct cycle types) proves the claim.
\end{proof}
The following is well known.
\begin{lem} Let $G$ be any group. Then, for any $\psi\in\Aut(G)$ and $\phi \in \Inn(G)$, we have that $R(\psi\phi)=R(\psi)$.
\end{lem}
\begin{lem}\label{autequalsinn} Let $G$ be a group with subgroup $\FAlt(X)$, where $X$ is an infinite set, and with $\Aut(G)=\Inn(G)$. Then $G$ has the $R_\infty$ property.
\end{lem}
\begin{proof} Let $\phi\in \Aut(G)$. By assumption $\Aut(G)=Inn(G)$. Therefore, by the previous lemma, $R(\phi)=R(\id_G)$. Now $R(\id_G)=\infty$ by Lemma \ref{innerautsfine}.
\end{proof}
Lemma \ref{autequalsinn} implies that, for any infinite set $X$, $\Sym(X)$ has the $R_\infty$ property.
\begin{not*} For any $g \in \Sym(X)$ and $x \in X$, let $\mathcal{O}_x(g):=\{xg^d : d \in \Z\}$. Also, let $\eta_r(g):=\quotient{|\{x \in X : |\mathcal{O}_x(g)|=r\}|}{r}$, the number of $r$-cycles in $g$. We shall use $\eta_1(g)$ to denote the number of fixed points of $g$ and $\eta_{\infty}(g)$ to denote the number of distinct infinite orbits induced by $g$. If any of these values is infinite then, since our arguments will be unaffected by the size of this infinity, we shall write $\eta_r(g)=\infty$.
\end{not*}
From the previous section, for any group fully containing $\FAlt(X)$ we have that the map $\Psi: N_{\Sym(X)}(G) \rightarrow \Aut(G), \rho\mapsto \phi_\rho$ is an isomorphism. We may therefore consider elements of $\Aut(G)$ as elements of $\Sym(X)$.
\begin{prop}\label{firstresult} Let $G$ fully contain $\FAlt(X)$, where $X$ is an infinite set, and let $\rho \in N_{\Sym(X)}(G)$. If there is an $r \in \N$ such that $\eta_r(\rho)<\infty$, then $R(\phi_\rho)=\infty$.
\end{prop}
\begin{proof} We shall work with the reformulation of twisted conjugacy in (\ref{eqnequivalenttwistedconj}) above and argue for any $\rho \in N_{\Sym(X)}(G)$ using three cases. Let $s \in \N$ be the smallest number such that $\eta_s(\rho)$ is finite.

\underline{Case A:} $s=1$ and $\eta_\infty(\rho)>0$. As with the proof of Proposition \ref{prop-faltchar}, let $x_0$ lie in an infinite cycle of $\rho$ and, for each $i \in \Z$, let $x_i:=x_0g^i$. For each $k \in \N$, let
$$
a_k:=\prod\limits_{i=0}^{k-1} (x_{2i}\;\;x_{2i+1}).
$$
The set of elements lying in disjoint $\phi_\rho$-twisted conjugacy classes is then given by $\{a_{2k}\mid k \in \N\}\subset \FAlt(X)$. This is because $\eta_1(\rho a_k)$ is finite for all $k\in \N$, and is strictly  increasing as a function of $k$. Thus, if $i\ne j$, the elements $\rho a_i$ and $\rho a_j$ have a different number of fixed points and hence are not conjugate in $G\leqslant\Sym(X)$.

\underline{Case B:} $s=1$ and $\eta_\infty(\rho)=0$. Since $\rho$ has finitely many fixed points and no infinite cycles, $\rho$ contains infinitely many finite cycles. Thus $\rho$ has infinitely many odd length cycles or infinitely many even length cycles. First assume that $\rho$ has infinitely many odd length cycles and index a countably infinite subset of these by the natural numbers. Let $\rho=\rho'\prod_{i \in \N}\rho_i$, where each $\rho_i$ is a finite cycle of odd length and $\rho'\in\Sym(X)$ has cycles with disjoint support from all of the $\rho_i$'s. Now, for any $m \in \N$, $\rho (\rho_m)^{-1}$ has more fixed points than $\rho$. Defining 
$$
a_k:=\prod\limits_{i=1}^{k} \rho_i^{-1} \in \FAlt(X)
$$
means that $i<j \Rightarrow \eta_1(\rho a_i)<\eta_1(\rho a_j)$ and so $\{a_k \mid k \in \N\}$ provides our infinite family of elements which are pairwise not $\phi_\rho$-twisted conjugate. Similarly, if $\rho$ has infinitely many even length cycles, complete the same construction with $\rho=\rho'\prod_{i \in \N}\rho_i$ where each $\rho_i$ is a finite cycle of $\rho$ of even length and $\rho'\in \Sym(X)$ has cycles with disjoint support from all of the $\rho_i$'s. Then  $\{a_{2k} \mid k \in \N\}$ provides the infinite family in $\FAlt(X)$.

\underline{Case C:} $s>1$. All we shall use is that $\rho$ has infinitely many fixed points. For any $k \in \N$, let $a_k$ consist of $2k$ $s$-cycles such that $\Supp(a_k)\subset X \setminus \Supp(\rho)$. We then have, for all $k \in \N$: that $a_k \in \FAlt(X)$; that $\eta_s(\rho a_k)$ is finite; and that $\eta_s(\rho a_k)$
is strictly increasing as a function of $k$.
\end{proof}

\begin{prop}\label{prop-conjcond} Let $a, b \in \Sym(X)$, $\Supp(b)\subsetneq \Supp(a)$, and $g \in \Sym(X)$ satisfy $g^{-1}ag=b$. Then $\eta_\infty(g)>0$.
\end{prop}
\begin{proof} We assume, for a contradiction, that $\eta_\infty(g)=0$. Since $g^{-1}ag=b$, $g$ must restrict to a bijection from $\Supp(a)$ to $\Supp(b)$ i.e.\
$$(\Supp(a)\cup\Supp(b))\setminus(\Supp(a)\cap \Supp(b))\subseteq\Supp(g)$$
which from our hypotheses implies that
$$\Supp(a)\setminus \Supp(b)\subseteq\Supp(g)$$
where $\Supp(a)\setminus \Supp(b)\ne\emptyset$ since $\Supp(b)\ne \Supp(a)$. Thus $g$ sends some $n \in \Supp(a)\setminus\Supp(b)$ to some $m \in \Supp(b)$. Now, since all of the cycles in $g$ are finite, there is a $k \in \N$ such that $(n)g^k=n$. Therefore $g$ sends a point in $\Supp(b)$ to a point in $X\setminus\Supp(b)$. This would mean that $\Supp(g^{-1}ag)\cap (X\setminus\Supp(b))\ne \emptyset$ and that $g^{-1}ag$ and $b$ have different supports, a contradiction.
\end{proof}
\begin{cor}\label{prop-infinitecycleneeded} Let $G$ be a group fully containing $\FAlt(X)$. If $\eta_\infty(g)=0$ for all $g \in G$, then $G$ has the $R_\infty$ property.
\end{cor}
\begin{proof} By Proposition \ref{firstresult}, if $\phi_\rho \in \Aut(G)$ has $\eta_s(\rho)<\infty$ for some $s \in \N$, then $R(\phi_\rho)=\infty$. We may therefore assume that $\eta_r(\rho)=\infty$ for all $r \in \N$. This implies that $X\setminus \Supp(\rho)$ is an infinite set.

Our aim is to show that there is an infinite set of elements in $G$ which are not $\phi_\rho$-twisted conjugate. Let $b_0:=1$, the identity element of $G$. For each $k \in \N$, let $b_k:=b_k'b_{k-1}$ where $\eta_2(b_k')=2$, $|\Supp(b_k')|=4$, $\Supp(b_k')\subset X\setminus \Supp(\rho)$, and $\Supp(b_k')\cap\Supp(b_{k-1})=\emptyset$. Thus, for each $k \in \N$, $b_k \in \FAlt(X)$ and $\eta_2(b_k)=2k$. If $i<j$, then $\Supp(b_i)\subsetneq\Supp(b_j)$ and so $\Supp(\rho b_i)\subsetneq\Supp(\rho b_j)$. Since $\eta_\infty(g)=0$ for all $g \in G$, Proposition \ref{prop-conjcond} implies that not two elements in $\{\rho b_k\mid k \in \N\}$ are conjugate in $G$ i.e.\ $R(\phi_\rho)=\infty$.
\end{proof}
Notice that this provides an alternative proof to \cite{rinfinityHou} and \cite{rinfinityHou2} that $\FSym(X)$ has the $R_\infty$ property. We also have the following.
\begin{cor}\label{cor-torsion}
Let $G$ be an infinite torsion group which fully contains $\FAlt(X)$. Then $G$ has the $R_\infty$ property. 
\end{cor}
\begin{lem}\label{torFAlt} Let $G\le \Sym(X)$ be torsion. Then $\langle G, \FAlt(X)\rangle$ is also torsion.
\end{lem}
\begin{proof} Consider an element $\sigma g$ where $\sigma \in \FSym(X)$ and $g \in G$. It suffices to show that $\sigma g$ is torsion. Fix a $k\in \N$ such that $\Supp(\sigma)$ has trivial intersection with all orbits of $g$ of size $r>k$. Since $G$ is torsion, the order of $g$ is one such $k$ (but since $|\Supp(\sigma)|<\infty$, we could find a $k$ using only the assumption that all orbits of all elements of $G$ are finite). Let
\[F:=\bigcup\limits_{1\leqslant i\leqslant k}\Supp(g^{-i}\sigma g^i).\]
Now $|F|\le k|\Supp(\sigma)|<\infty$, $\sigma g$ restricts to a bijection on $X\setminus F$, the elements $\sigma g$ and $g$ induce the same permutation on $X\setminus F$, and the orbits of $g$ within $X\setminus F$ are of bounded size (since otherwise $g$ cannot be torsion). Hence the orbits of $\sigma g$ within $X$ are of bounded size and so $\sigma g$ is torsion.
\end{proof}
\begin{rem}\label{finiteorbitsremainfinite} A consequence of this proof is that if all elements of $G\le \Sym(X)$ have only finite orbits, then the elements of $\langle G, \FAlt(X)\rangle$ also only have finite orbits.
\end{rem}
\begin{cor} Let $G$ be a torsion group. For every infinite $\alpha\ge |G|$, there exists a torsion group $H_\alpha$ of cardinality $\alpha$ which has the $R_\infty$ property and contains an isomorphic copy of $G$.
\end{cor}
\begin{proof} Let $G$ be torsion, $\hat{G}$ denote the right regular representation of $G$, and let $\alpha \ge G$ if $G$ is infinite and $\alpha\ge |\Z|$ otherwise. Then there is a set $Y_\alpha$ such that $|Y_\alpha |=\alpha$. Also $\hat{G}\le \Sym(G) \hookrightarrow \Sym(G\sqcup Y_\alpha)$ via the natural inclusion of the set $G$ into the set $G\sqcup Y_\alpha$. Let $G_\alpha$ denote the image of $\hat{G}$ in $\Sym(G\sqcup Y_\alpha)$, using the restriction of this map. Now $H_\alpha:=\langle G_\alpha, \FAlt(G\sqcup Y_\alpha)\rangle$ has cardinality $\alpha$. Moreover it is torsion by Lemma \ref{torFAlt} and so has the $R_\infty$ property by Corollary \ref{cor-torsion}.
\end{proof}

There are also groups which are not torsion and have no infinite cycles. Consider an element $\rho \in \Sym(X)$ with $\eta_r(\rho)$ non-zero for infinitely many $r \in \N$. Then $\rho$ has infinite order, but need not contain an infinite cycle. Therefore $\rho$  generates an infinite cyclic group, but $\langle \rho, \FSym(X)\rangle$ is not finitely generated. This is an interesting example since $\FSym(\Z)\rtimes \Z$, which also consists of the group $\FSym$ together with a single element of infinite order, is 2-generated (being the second Houghton group $H_2$). For another example, consider $G=\prod_{i\in\N} C_2$. This can be seen as a subgroup $G_1$ of $\Sym(\{1,2\}\times\N)$ where the $i$th $C_2$ transposes the points $(1,i)$ and $(2, i)$ and fixes all other points of $\{1,2\}\times\N$. Now $\prod_{i\in\N} C_2\cong \bigoplus_{i \in \R}C_2$ (both are vector spaces of rank $|\R|$ over $F_2$) and so $G$ can also be seen as a subgroup $G_2$ of $\Sym(\{1,2\}\times\R)$ with generators $g_i$ (for each $i \in \R$) that transpose the points $(1,i)$ and $(2, i)$ and fix all other points of $\{1,2\}\times\R$. But $\langle G_1, \FAlt(\{1,2\}\times\N)\rangle\not\cong\langle G_2, \FAlt(\{1,2\}\times\R)\rangle$, since $\langle G_1, \FAlt(\{1,2\}\times\N)\rangle\le \Sym(\{1,2\}\times\N)$ and $\FAlt(\{1,2\}\times\R)$ does not embed into $\Sym(\{1,2\}\times\N)$ by \cite{faltembedintosym}.

\begin{cor}\label{cor-countabletorsion} There exist uncountably many countable torsion groups which have the $R_\infty$ property.
\end{cor}
\begin{proof} We will work within $\Sym(\N\times\N)$. For each $n\ge 2$, define
\[\phi^{(n)}:C_n\hookrightarrow \Sym(\N\times\N), (1\;\ldots\;n)\mapsto \rho_n\]
where $\Supp(\rho_n)=\{(m,n)\mid m\in \N\}$ and
\begin{align*}
(m,n)\rho_n:=\left\{\begin{array}{ll}(m-n+1, n)&\text{if}\;m\equiv0\bmod{n}\\(m+1, n)&\text{otherwise}\end{array}\right.
\end{align*}
i.e.\ $\rho_n$ consists of $n$-cycles `all the way along' the $n$\ts{th} copy of $\N$.

Let $\P$ denote the set of all prime numbers. Then, for any subset $S\subseteq\P$, let $G_S:=\bigoplus_{p \in S}C_p$. Note that there are uncountably many choices for $S$. Also,
\[\bigoplus\limits_{p \in S}C_p\hookrightarrow \Sym(\N\times\N)\]
by using the maps $\phi^{(n)}$ defined above. For any $S\subseteq\P$, let $\tilde{G}_S:=\langle G_S, \FAlt(\N\times\N)\rangle$, which fully contains $\FAlt(\N\times\N)$ and, by Lemma \ref{torFAlt}, is torsion. Hence Corollary \ref{cor-torsion} applies to $\tilde{G}_S$ and it has the $R_\infty$ property. Our final aim is therefore to show that if $S\ne S'$, then $\tilde{G}_S$ and $\tilde{G}_{S'}$ are not isomorphic. By Proposition \ref{prop-faltchar}, $\tilde{G}_S$ and $\tilde{G}_{S'}$ each have $\FAlt(\N\times\N)$ as a unique minimal normal subgroup. Since $G_S$ and $G_{S'}$ contain no non-trivial elements of finite support,
\begin{align*}\quotient{\tilde{G}_S}{\FAlt(\N\times\N)}\cong G_S\;\text{and}\;\quotient{\tilde{G}_{S'}}{\FAlt(\N\times\N)}\cong G_{S'}.\end{align*} 
Hence if $\tilde{G}_S$ and $\tilde{G}_{S'}$ are isomorphic, then $G_S$ and $G_{S'}$ are isomorphic. But since $S\ne S'$, there is a $p \in \P$ in one set that is not in the other. Without loss of generality let $p \in S\setminus S'$. By construction, $G_S$ has $p$-torsion but $G_{S'}$ does not. Hence $\tilde{G}_S\not\cong \tilde{G}_{S'}$.
\end{proof}

\begin{cor}\label{cor-countabletorsionfg} There exist uncountably many finitely generated torsion groups which have the $R_\infty$ property.
\end{cor}
\begin{proof} In \cite{tarskimonster} the Tarski monsters, an uncountable family of finitely generated infinite p-groups, are described. Let $M_1$ and $M_2$ be non-isomorphic Tarski monsters. For any group $G$, let $\hat{G}$ denote the right regular representation of $G$ and let $\tilde{G}:=\langle \hat{G}, \FSym(G)\rangle$. By Lemma \ref{torFAlt}, $\tilde{M_1}$ and $\tilde{M_2}$ are torsion. By \cite[Prop 5.10]{osintransitive}, $\tilde{M_1}$ and $\tilde{M_2}$ are finitely generated. Moreover $\tilde{M_1}\not\cong \tilde{M_2}$ since they are both monolithic (by Proposition \ref{prop-faltchar}) but if we quotient by this unique minimal normal subgroup then we obtain non-isomorphic groups.
\end{proof}
There are many equivalent definitions of the following.
\begin{defn} A group $G$ is residually finite if for each non-trivial element $g\in G$ there exists a finite group $F_g$ and a homomorphism $\phi_g:G\rightarrow F_g$ such that $(g)\phi_g\ne 1$.
\end{defn}
It is the following well known reformulation that shall be of use to us.
\begin{lem} A group $G$ is residually finite if and only if it can be embedded inside the direct product of a family of finite groups. Moreover the family comprises of the finite quotients of $G$.
\end{lem}
\begin{cor}\label{cor-residuallyfinite} Let $G$ be an infinite residually finite group, and $X$ be the union of the finite quotients of $G$. Then, using this action, $\langle G, \FAlt(X)\rangle$ has the $R_\infty$ property.
\end{cor}
\begin{proof} Since $G$ is residually finite, it can be embedded inside the direct product of a family of finite groups (which are those groups appearing as finite quotients of $G$). Therefore any element $g \in G$ has only finite orbits, and by Remark \ref{finiteorbitsremainfinite}, any element in $\langle G, \FAlt(X)\rangle$ also only has finite orbits. Hence Corollary \ref{prop-infinitecycleneeded} applies, and $\langle G, \FAlt(X)\rangle$ has the $R_\infty$ property.
\end{proof}

\section{The $R_\infty$ property and commensurable groups}
This final section involves results for commensurable groups.
\begin{not*} Let $N\unlhd_f G$ denote that $N$ is normal and finite index in $G$.
\end{not*}
\begin{defn} Let $G$ and $H$ be groups. We say that $G$ is commensurable to $H$ if and only if there exist $N_G\cong N_H$ with $N_G\unlhd_f G$ and $N_H\unlhd_f H$. Note that if $G$ is finitely generated then so is $N_G$, which then implies that $H$ is finitely generated.
\end{defn}
We will work towards Theorem \ref{thm-Hn} which applies to the Houghton groups, a family of groups $H_n$ indexed over $\N$ where, for each $n \in \N$, $H_n$ acts on a set $X_n$ and $\FSym(X_n)\leqslant H_n\leqslant \Sym(X_n)$. Each group $H_n$ therefore fully contains $\FAlt(X_n)$. These were first introduced in \cite{Houghton}, but we rely heavily on \cite{Cox} where an introduction to these groups can be found and a description, for all $n\ge2$, of the structure of the automorphism group for all finite index subgroups of $H_n$ is given. We start with three well known results.

\begin{lem}\label{lem-findnorm} If $H\leqslant_fG$, then $\exists\;N\leqslant_f H$ which is normal in $G$.
\end{lem}
\begin{proof} Let $H$ have index $n$ in $G$ and let $N:=\bigcap_{g \in G}(g^{-1}Hg)$. Then $G$ acts on $H\backslash G$ by right multiplication, and so there is a homomorphism $\phi: G\rightarrow S_n$. Now $h \in \ker(\phi)$ if,
\begin{align*}
&Hgh=Hg\;\text{for all}\;g \in G\\
\Leftrightarrow&ghg^{-1}\in H\;\text{for all}\;g \in G\\
\Leftrightarrow&h \in g^{-1}Hg\;\text{for all}\;g \in G.
\end{align*}
Hence $\ker(\phi)=N$ and $N$ is normal. Moreover $\quotient{G}{\ker(\phi)}\cong \textrm{Im}(\phi)\le S_n$, and so $N$ has index $m$ in $G$ where $m\le n!$ and $m$ divides $n!$.
\end{proof}

\begin{lem}\label{lem-findchar} If $H\leqslant_fG$ and $G$ is finitely generated, then $\exists\;K\leqslant_f H$ which is characteristic in $G$.
\end{lem}
\begin{proof} Suppose $H\leqslant_nG$. We first show that there exist only finitely many subgroups of $G$ of a given index. As in the previous lemma, right multiplication by $G$ on $H\backslash G$ gives a homomorphism $\phi_H: G\rightarrow S_n$. Note that $\Stab(H)=H$ since $g \in \Stab(H)\Leftrightarrow Hg=H$. Thus, by choosing $1 \in \Z_n$ to correspond to the coset $H$ in $H\backslash G$, the preimage of $\Stab(1)$ in $S_n$ is $H$. Hence $H=H'\Leftrightarrow \phi_H=\phi_{H'}$.
\begin{align*}
\;\text{But}\;G\;\text{finitely generated}\;\Rightarrow \exists\;\text{only finitely many homomorphisms}\; G\rightarrow S_n
\end{align*}
(there are $(n!)^{|S|}$ maps from $S$ to $S_n$) and so there can only be finitely many index $n$ subgroups. Now let
 \begin{align}
K:=\bigcap_{\phi \in \Aut(G)}\!\!\!\!\!(H)\phi\label{factforlemma3}
 \end{align}
and note that, for any $\phi \in \Aut(G)$, $(H)\phi\leqslant_nG$. But there are only finitely many possible images for $H$ in (\ref{factforlemma3}), and so (since the intersection of finitely many subgroups of finite index is of finite index) $K$ is finite index in $G$. Finally, $K$ is characteristic in $G$  since the image of $K$ under $\psi \in \Aut(G)$ is contained within
\begin{align*}
 \bigcap_{\phi \in \Aut(G)}\!\!\!\!\!((H)\phi\psi)
\end{align*}
which is equal to $K$.
\end{proof}

\begin{lem}\cite[Lem 2.2(ii)]{rinfinityextensions}\label{lem-finiteextensions} Let $D$ be a group with the $R_{\infty}$ property and
 $$
1\longrightarrow D \longrightarrow  E \longrightarrow F \longrightarrow 1
 $$
be a short exact sequence of groups. If $D$ is characteristic in $E$ and $F$ is any finite group, then $E$ has the $R_{\infty}$ property.
\end{lem}
Combining the previous two results provides an easier condition to check in order to show that all commensurable groups have the $R_\infty$ property.
\begin{lem}\label{prop-finiteindexcommensurable}
Let $G$ be a finitely generated group. If $G$ and all finite index subgroups of $G$ have the $R_\infty$ property, then all groups commensurable to $G$ have the $R_\infty$ property.
\end{lem}
\begin{proof} Let $H$ be commensurable to $G$. Then $\exists\;N\unlhd_fG, H$. By Lemma \ref{lem-findchar}, there exists a group $U$ which is characteristic in $H$ and such that $U\leqslant_fG, H$. From our assumption that all finite index subgroups of $G$ have the $R_\infty$ property, $U$ has the $R_\infty$ property. Hence, by Lemma \ref{lem-finiteextensions}, $H$ has the $R_\infty$ property.
\end{proof}

Our final aim is the following. Although it was done independently, our argument has similarities to \cite[First proof of Thm. 3.8]{rinfinityHou2}. Their argument produces elements of different orders, whilst we produce elements of different cycle types. The flexibility that this affords allows our arguments to generalise from $H_n$ to certain subgroups $U_p\le H_n$.

\begin{thm} \label{thm-Hn} Let $n\in \N$. If $G$ is any group commensurable to $H_n$, the $n$\ts{th} Houghton group, then $G$ has the $R_\infty$ property.
\end{thm}

\begin{proof} We first work with $\FAlt$. If $G$ is commensurable to $\FAlt(X)$, then there exists $N\unlhd_f\FAlt(X), G$. Now, since $\FAlt(X)$ is simple and infinite, $N=\FAlt(X)$. Hence we have the short exact sequence
\[1\longrightarrow \FAlt(X) \longrightarrow  G \longrightarrow F \longrightarrow 1\]
where $F$ is some finite group. Let $\phi \in \Aut(G)$ and consider $\FAlt(X)\cap (\FAlt(X))\phi$. This has finite index in $\FAlt(X)$. Using Lemma \ref{lem-findnorm} and that $\FAlt(X)$ is simple, we have $(\FAlt(X))\phi=\FAlt(X)$ i.e.\ that $\FAlt(X)$ is characteristic in $G$. Since $\FAlt(X)$ is torsion, Corollary \ref{cor-torsion} states that it has the $R_\infty$ property. Hence Lemma \ref{lem-finiteextensions} applies to $G$ implying that $G$ has the $R_\infty$ property.

We now work with $n\ge 2$. From Lemma \ref{prop-finiteindexcommensurable}, it is sufficient to show that, for any $n\ge 2$, all finite index subgroups of $H_n$ have the $R_\infty$ property.

Fix an $n\ge2$. There are a family of finite index, characteristic subgroups of $H_n$ defined in \cite{Hou2} and denoted $U_p$ where $p \in \N$. They showed that, for any $U\le_fH_n$, there exists a $p\in \N$ such that $U_p\le_fU$. This was strengthened in \cite[Prop. 5.12]{Cox} by showing that, for any $U\leqslant_fH_n$, there exists an $m \in \N$ such that $U_m\le_fU$ and
\[\Aut(U)\hspace{0.1cm} _\Psi\hspace{-0.1cm}\cong N_{\Sym(X_n)}(U)\leqslant N_{\Sym(X_n)}(U_m)\cong_\Psi \Aut(U_m)\]
where $\Psi: N_{\Sym(X_n)}(G)\mapsto \Aut(G)$ is defined by $(g)\Psi=\phi_g$. Furthermore, by \cite[Lem. 5.9]{Cox}, there is a monomorphism $\mu: N_{\Sym(X_n)}(U_m)\hookrightarrow N_{\Sym(X_{nm})}(H_{nm})$ and, for any $k\ge2$, $N_{\Sym(X_k)}(H_k)=H_k\rtimes S_k$. Importantly, this monomorphism preserves cycle type. We shall apply Proposition \ref{firstresult} to show that any group with automorphism group contained within $N_{\Sym(X_k)}(H_k)$ for some $k\ge2$ has the $R_\infty$ property.

Fix a $k\ge2$. Notice that for all $r \in \N\setminus\{1\}$ and for all $g \in H_k$, $\eta_r(g)$ is finite. Given a $\rho \in H_k\rtimes S_k$, which is isomorphic to $\Aut(H_k)$ via the map $\rho\mapsto \phi_\rho$, we have that $\eta_r(\rho)$ is infinite if and only if $\rho$ induces a cyclic permutation of $r$ branches of $X_k$. Thus, for all $\rho \in N_{\Sym(X_k)}(H_k)$ and all $r>k$ we have that $\eta_r(\rho)$ is finite. Now, for any $U\leqslant_fH_n$, there exists an $m\in \N$ such that $N_{\Sym(X_n)}(U)\leqslant N_{\Sym(X_n)}(U_m)$. Consider if $\rho \in N_{\Sym(X_n)}(U_m)$. Using the above homomorphism $\mu: N_{\Sym(X_n)}(U_m)\rightarrow N_{\Sym(X_{nm})}(H_{nm})$, we have that $\eta_r((\rho)\mu)$ is finite for all $r>nm$. Since $\mu$ preserves cycle type, $\eta_r(\rho)$ is also finite for all $r>nm$. Hence, by Proposition \ref{firstresult}, $R(\phi_\rho)=\infty$ and so all automorphisms of $U$ have infinite Reidemeister number. Thus all finite index subgroups of $H_n$ have the $R_\infty$ property and so Lemma \ref{prop-finiteindexcommensurable} yields the result.
\end{proof}

\def\cprime{$'$}
\providecommand{\bysame}{\leavevmode\hbox to3em{\hrulefill}\thinspace}
\providecommand{\MR}{\relax\ifhmode\unskip\space\fi MR }
\providecommand{\MRhref}[2]{%
  \href{http://www.ams.org/mathscinet-getitem?mr=#1}{#2}
}
\providecommand{\href}[2]{#2}

\end{document}